\documentclass[11pt, reqno]{amsart}

\calclayout

\usepackage[norsk,english]{babel}
\usepackage[left]{lineno}
\usepackage[utf8]{inputenc}
\usepackage{amsthm}
\usepackage{amsmath}
\usepackage{amssymb}
\usepackage{mathtools}
\usepackage{xcolor}
\usepackage{enumitem}
\usepackage{hyperref}
\usepackage[capitalize]{cleveref}
\usepackage{cite}

\crefname{equation}{eq.}{eqs.}
\Crefname{equation}{Equation}{Equations}

\hypersetup{
  colorlinks   = true, 
  urlcolor     = blue, 
  linkcolor    = blue, 
  citecolor   = red 
}

\newcommand{\jpb}[1]{\langle #1 \rangle}
\newcommand{\norm}[1]{\| #1 \|}
\newcommand{\abs}[1]{\left| #1 \right|}

\DeclareMathOperator\supp{supp}

\newtheorem{theorem}{Theorem}[section]
\newtheorem{lemma}[theorem]{Lemma}

\newtheorem{proposition}[theorem]{Proposition}
\theoremstyle{remark}

\numberwithin{equation}{section}

\title[Solitary waves in equations with nonlocal cubic terms]{Remarks on solitary waves in equations with nonlocal cubic terms}
\author{Johanna Ulvedal Marstrander\(^1\)}
\address{\(^1\)Department of Mathematical Sciences, NTNU -- Norwegian University of Science and Technology, 7491 Trondheim, Norway}
\email{johanna.u.marstrander@ntnu.no}
\keywords{solitary wave, nonlinear dispersive equation, nonlocal equation, water wave, calculus of variations}
\subjclass{76B15, 76B25, 35A15}
\thanks{The author acknowledges the support of the project IMod (Grant No. 325114) from the Research Council of Norway.
Part of this work was carried out during the workshop MaGIC 2023. 
}
    
\begin{document}

\begin{abstract}
In this overview paper, we show existence of smooth solitary-wave solutions to the nonlinear, dispersive evolution equations of the form 
\begin{equation*}
    \partial_t u + \partial_x(\Lambda^s u + u\Lambda^r u^2) = 0,
\end{equation*}
where \(\Lambda^s, \Lambda^r\) are Bessel-type Fourier multipliers. The linear operator may be of low fractional order, \(s>0\), while the operator on the nonlinear part is assumed to act slightly smoother, \(r<s-1\). The problem is related to the mathematical theory of water waves; we build upon previous works on similar equations, extending them to allow for a nonlocal nonlinearity. Mathematical tools include constrained minimization, Lion's concentration--compactness principle, spectral estimates, and product estimates in fractional Sobolev spaces. 
\end{abstract}

\maketitle
\section{Introduction}
\label{sec:introduction}
Many one-dimensional model equations for water waves can be written as
\begin{equation}
    \partial_tu + \partial_x(L u - n(u)) = 0,
    \label{eq:model_general}
\end{equation}
where \(u\) is a real-valued function of time and space, \(L\) is a Fourier multiplier and \(n\) is a local, nonlinear function. This class of equations includes many well-known models, such as the Korteweg--de Vries and Benjamin--Ono equations when the operator \(L\) is a positive-order operator, and Whitham-type equations when \(L\) has negative order \cite{linares2014, ehrnstrom2023}. We restrict our attention here to operators \(L\) of positive order, which physically corresponds to situations with surface tension and stronger dispersion \cite{lannes2013}. 

In the present work, we consider a way to include frequency interaction also in the nonlinear term. This is natural in the context of water waves, since exact modeling may give rise to nonlocal nonlinearities \cite{ionescu2018,lannes2013}. Recent years have seen an increased interest in these types of equations, see e.g. \cite{dinvay2021,duchene2018,orke2022} concerning traveling waves. 

There has been extensive research on whether or not \cref{eq:model_general} under different assumptions on \(L\) and \(n\) admits a type of solutions called \emph{solitary waves}. These are localized solutions that propagate at constant speed \(\nu\) while retaining their original shape: 
\begin{equation*}
    (x,t)\mapsto u(x-\nu t) \quad\text{and}\quad u(x-\nu t)\to 0 \text{ as } \abs{x-\nu t}\to \infty.
\end{equation*}
While their existence may be demonstrated in several ways \cite{buffoni2004,groves2011,stefanov2020,arnesen2023}, many constructions of solitary waves are based upon Weinstein's argument in\cite{weinstein1987}, where solutions are found by \(L^2\)-constrained minimization with the help of Lions' concentration--compactness method \cite{lions1984}, see e.g. \cite{maehlen2020} and references therein. Building upon \cite{arnesen2016} and \cite{maehlen2020}, we shall investigate how this technique may be modified to accomodate nonlocal nonlinearities. A variation of Weinstein's argument adapted to negative \(s\), see \cite{ehrnstrom2012}, has later been developed and applied to bidirectional systems with negative-order Fourier multipliers in nonlinear terms \cite{duchene2018,dinvay2021}. Nonlocal nonlinearities also arise in works where solitary waves are constructed using a local variational reduction \cite{groves2008,ehrnstrom2022a}. In this paper, however, we will study nonlocal nonlinearities in the one-dimensional setting for positive \(s\).

Solitary-wave solutions to \cref{eq:model_general} arise when dispersive effects due to the term \(Lu\) and nonlinear effects due to \(n(u)\) are balanced \cite{lannes2013, linares2014}. Typically, one can consider operators \(L\) of order \(s\), and nonlinearities with leading-order power \(p\), and show that solitary-wave solutions to \cref{eq:model_general} exist for a range of parameters \(s, p\). For positive-order operators \(L\), existence is shown for a given \(p\) if \(s\) is sufficiently high, \(s>\frac{p-1}{p+1}\) \cite{arnesen2016, maehlen2020}. Contrariwise, the fractional KdV equation does not admit solitary-wave solutions for \(s\) below this limit \cite{linares2014}. Thus, if one wishes to find solitary waves when the dispersion is weak, the nonlinearity must be correspondingly weak. 

In our case, instead of considering a range of powers \(p\), we fix a cubic nonlinearity and apply an operator of order \(r\). Let \(\Lambda^{\alpha}\) be the Bessel Fourier multiplier in space defined by
\begin{equation*}
    \widehat{\Lambda^{\alpha} u}(\xi) = (1+\xi^2)^{\frac{\alpha}{2}}\hat{u}(\xi).
\end{equation*}
For real numbers \(s, r\), we shall consider the equation
\begin{equation}
    \partial_t u + \partial_x (\Lambda^s u - u\Lambda^r u^2) = 0.
    \label{eq:time_dependent}
\end{equation}
It is also possible to replace \(\Lambda^s, \Lambda^r\) by more general operators, see \cref{sec:general_assumptions}. A motivation for studying the cubic nonlinearity \(u\Lambda^r u^2\) is that a variational formulation is readily available. This is not the case for the quadratic equivalent \(u\Lambda^r u\). To use a similar approach for an equation with a quadratic nonlocal term, one needs to consider pseudo-products. This was recently done by the author in \cite{marstrander2023}. The author is grateful to Douglas Svensson Seth for pointing out helpful relationships in the variational formulation.

Observe that when \(r=0\), \cref{eq:time_dependent} coincides with \cref{eq:model_general} with \(L=\Lambda^s\) and \({n(u)= u^3}\).
Note that the introduction of an operator in the nonlinear term affects the interplay between dispersive and nonlinear effects.  Just as the power of the nonlinearity \(p\) is bounded above by the order of the dispersive operator \(s\) in \cite{arnesen2016,maehlen2020}, \(r\) must be bounded above by \(s\) in our work. 
 We shall assume that \(s,r\in \mathbb{R}\) satisfy
\begin{equation}
    s>0,\qquad r<s-1.\label{eq:assump_sr}
\end{equation}
{Here, \(r\) may be positive or negative.} When \(r=0\), the lower bound for \(s\) that follows agrees with \(p<2s+1\) from \cite{arnesen2016}. 
However,  if \(r\) is sufficiently negative, we will see that there are solitary-wave solutions to \cref{eq:time_dependent} for all \(s>0\). Thus the existence of solitary waves in this type of equation does not depend directly on the power of the nonlinearity, but rather on the strength of the nonlinearity in a general sense. 
Under the solitary-wave ansatz, \cref{eq:time_dependent} becomes
\begin{equation}
    -\nu u + \Lambda^s u - u\Lambda^r u^2 = 0.
    \label{eq:solitary_wave}
\end{equation}

\begin{theorem}
{Let \(s>0,\, r<s-1\).} For every \(\mu>0\), there is a \(u\in H^{\infty}(\mathbb{R})\) satisfying \(\frac{1}{2}\norm{u}_{L^2}^2 = \mu\) that solves the solitary-wave equation \cref{eq:solitary_wave}. The corresponding solitary waves are subcritical, that is, the wave speed satisfies \(\nu<1\). {Furthermore, for any fixed \(\mu_{0}>0\), the estimates
\begin{equation}
    \norm{u}_{L^{\infty}}\lesssim \norm{u}_{H^{\frac{s}{2}}}\eqsim \mu^{\frac{1}{2}}, \qquad 1-\nu \eqsim \mu^2.\label{eq:small_solutions_est}
\end{equation}
hold uniformly in \(\mu\in(0,\mu_{0})\).
}
\label{thm:main}
\end{theorem}
{ We show existence of smooth solitary waves of any size, small or large.} Here, the space \({H^{\infty}(\mathbb{R}) = \cap_{t\in \mathbb{R}}H^{t}(\mathbb{R})}\), where \(H^{t}(\mathbb{R})\) are fractional Sobolev spaces, see \cref{sec:variatonal_formulation}. 
 { In \cref{eq:small_solutions_est}, the implicit constants may depend on \(\mu_{0}\), but \(\mu_{0}\) does not need to be small.}

While the upper bound on \(r\) in \cref{eq:assump_sr} appears in the proof through embedding and interpolation theorems, it is in fact related to properties of the equation and we cannot expect to find solutions below \(s>r+1\). When \(r=0\) we get the special case of \cref{eq:model_general} under the solitary-wave ansatz that is most similar to our equation,
\begin{equation}
    -\nu u + \Lambda^s u - u^3 =0.
    \label{eq:model}
\end{equation}
In this case, our assumptions on \(r,s\) imply that \(s>1\). For the generalized fKdV-equation with \(p=3\), the value \(s=1\) is the critical exponent below which solitary-wave solutions are no longer stable \cite{bona1987}, and the method we use yields a set of minimizers that form stable sets of solutions \cite{albert1999,arnesen2016} As mentioned, existence of solutions to \cref{eq:model} for \(s>\frac{p-1}{p+1} =\frac{1}{2}\) was shown in \cite{arnesen2016}, but then by appealing to a scaling argument at the cost of stability.

{
To keep the exposition clear and highlight what is new, our strategy in proving \cref{thm:main} will be to outline the method for \cref{eq:model} and discuss the changes that are made to accommodate the nonlocal nonlinearity in \cref{eq:solitary_wave}. This constitutes \cref{sec:outline}. \Cref{eq:model} is treated in \cite{arnesen2016} and we mostly follow that paper. We will, however, also refer to specific lemmas in \cite{maehlen2020} where some arguments have been considerably simplified. We prove rigorously only the parts that are new for \cref{eq:solitary_wave}. 
In \cref{sec:general_assumptions}, we discuss possible generalizations.
}

\section{Solitary-waves solutions to \cref{eq:model,eq:solitary_wave}}
\label{sec:outline}

\subsection{Formulation as a constrained minimization problem.}
\label{sec:variatonal_formulation}
In the classical case, existence of solutions to \cref{eq:model} is established via a related constrained minimization problem. Let \(\mathcal{S}'(\mathbb{R})\) be the space of tempered distributions{, and let \(L^{p}( \mathbb{R}), 1\leq p<\infty\) be the standard Lebesgue spaces with norm \(\left\|u \right\|_{L^{p}}^p = \int_{\mathbb{R}} \left|u \right|^p \,dx\). We search for minimizers in fractional Sobolev spaces, 
\begin{equation*}
    H^{t}(\mathbb{R}) = \{u\in \mathcal{S}'(\mathbb{R})\colon \norm{u}_{H^{t}} = \norm{\Lambda^t u}_{L^2} <\infty\}.
\end{equation*}
} Define the functionals \({\mathcal{E}, \mathcal{Q}, \mathcal{L}, \mathcal{N}\colon H^{\frac{s}{2}}(\mathbb{R}) \to \mathbb{R}}\) by 
\begin{equation*}
    \mathcal{E}(u)= \underbrace{\frac{1}{2}\int_{\mathbb{R}} u\Lambda^s u\,dx}_{=\mathcal{L}(u)} - \underbrace{\frac{1}{4}\int_{\mathbb{R}} u^4\,dx}_{=\mathcal{N}(u)} \quad \text{and} \quad    \mathcal{Q}(u)= \frac{1}{2}\int_{\mathbb{R}} u^2\,dx.
\end{equation*}
Due to the Lagrange multiplier principle, minimizers of the constrained minimization problem
\begin{equation*}
    \Gamma_{\mu} = \inf\{\mathcal{E}(u)\colon u\in H^{\frac{s}{2}}(\mathbb{R}) \text{ and }\mathcal{Q}(u)=\mu\}
\end{equation*}
will solve \cref{eq:model} with wave speed \(\nu\) being the Lagrange multiplier.

To formulate the corresponding minimization problem for \cref{eq:solitary_wave}, we simply replace \(\mathcal{N}\) with \(\tilde{\mathcal{N}}\) incorporating the nonlocal operator: 
\begin{equation*}
    \begin{split}
    \tilde{\mathcal{E}}(u)= \underbrace{\frac{1}{2}\int_{\mathbb{R}} u\Lambda^s u\,dx}_{=\mathcal{L}(u)} - \underbrace{\frac{1}{4}\int_{\mathbb{R}} u^2\Lambda^r u^2\,dx}_{=\tilde{\mathcal{N}}(u)},\\
    \tilde{\Gamma}_{\mu} = \inf\{\tilde{\mathcal{E}}(u)\colon u\in H^{\frac{s}{2}}(\mathbb{R}) \text{ and }\mathcal{Q}(u)=\mu\}
    \end{split}
\end{equation*}
It is easily verified that the Fréchet derivative of \(\tilde{\mathcal{N}}\) is \(u\Lambda^r u^2\),
and so a minimizer of \(\tilde{\Gamma}_{\mu}\) solves \cref{eq:solitary_wave} with langrange multiplier \(\nu\):
\begin{equation*}
    \mathcal{L}'(u)- \tilde{\mathcal{N}}'(u) - \nu\mathcal{Q}'(u) = -\nu u + \Lambda^{s}u - u\Lambda^r u^2 = 0.
\end{equation*}

\subsubsection{The infimum is well defined}
To show that \(\Gamma_{\mu}\) is bounded below, the strategy is to bound \(\mathcal{N}\) in terms of \(\mathcal{L}\) to a power \(\gamma<1\). Then we would have
\begin{equation*}
    \mathcal{E}(u) \geq \mathcal{L}(u)- \mathcal{\mathcal{N}}(u) \geq \mathcal{L}(u)- C\mathcal{L}(u)^{\gamma} > -\infty,
\end{equation*}
since \(\mathcal{L}(u)\) is positive and the first term dominates the last as \(\mathcal{L}(u)\to \infty\). 
This is straightforward since \(\mathcal{N} \eqsim \norm{u}_{L^4}^4\) and \(\mathcal{L}(u)\eqsim \norm{u}_{H^{\frac{s}{2}}}^2\) and Sobolev embedding and interpolation yields
\begin{equation}
    \norm{u}_{L^4}\lesssim \norm{u}_{H^{\frac{1}{4}}}\lesssim \norm{u}_{L^2}^{1-\frac{1}{2s}}\norm{u}_{H^{\frac{s}{2}}}^{\frac{1}{2s}},\label{eq:L4_bound_upper}
\end{equation}
and \(\frac{1}{2s}<\frac{1}{2}\) as \(s>1\).

{The same argument is directly applicable also to \(\tilde{\Gamma}_{\mu}\) as soon as we establish the inequality}
\begin{equation}
    \underbrace{\norm{u^2}_{H^{\frac{r}{2}}}^2}_{\eqsim\tilde{\mathcal{N}}(u)} \lesssim \norm{u}_{L^2}^{4-2\gamma}\underbrace{\norm{u}_{H^{\frac{s}{2}}}^{2\gamma}}_{\eqsim \mathcal{L}(u)^{\gamma}}\label{eq:N_upper}
\end{equation}
{for some \(0<\gamma < 1\)}. The nonlocal operator makes the argument more involved but has the added benefit of "weakening" the nonlinear part \(\tilde{\mathcal{N}}\) when \(r\) is sufficiently negative, allowing us to go below the limit in the classical case, \(s>1\). To show \cref{eq:N_upper}, we rely on product estimates in fractional Sobolev spaces, see \cref{prop:product_rule_hilbert} in \cref{appendix}. 

{
\begin{lemma}
    For \(u\in H^{\frac{s}{2}}(\mathbb{R})\), there is a \(\gamma<1\) such that
    \begin{equation*}
        \norm{u^2}_{H^{\frac{r}{2}}} \lesssim \norm{u}_{L^2}^{2-\gamma}\norm{u}_{H^{\frac{s}{2}}}^{\gamma}.
    \end{equation*}
    \label{lem:N_upper}
\end{lemma}

\begin{proof}
    We treat the cases \(s>1, s\leq 1\) separately and apply \cref{prop:product_rule_hilbert}.

    Suppose \(s>1\). We assume that \(r>0\). Otherwise, we can apply the same method to any \(\tilde{r}\) such that \(0<\tilde{r}<s-1\) and use that \(\norm{u^{2}}_{H^{\frac{r}{2}}}\leq \norm{u^{2}}_{H^{\frac{\tilde{r}}{2}}}\). We pick a \(\tau\in \mathbb{R}\) satisfying
    \[
    1 <\tau < s-r
    \]
    and apply \cref{prop:product_rule_hilbert} (i) with \(t_1 = \frac{r}{2}\) and \(t_2 = \frac{\tau}{2}\):
    \begin{equation*}
    \norm{u^{2}}_{H^{\frac{r}{2}}}\lesssim \norm{u}_{H^{\frac{r}{2}}}\norm{u}_{H^{\frac{\tau}{2}}}
    \end{equation*}
    Interpolation then yields the desired result,
    \begin{equation*}
        \norm{u^{2}}_{H^{\frac{r}{2}}}\lesssim \norm{u}_{L^2}^{2-\frac{r+\tau}{s}}\norm{u}_{H^{\frac{s}{2}}}^{\frac{r+\tau}{s}},
    \end{equation*}
    with \(\gamma = \frac{r+\tau}{s} <\frac{r+ s-r}{s} = 1\).

    Suppose now that \(s\leq 1\). If \(s=1\), we can pick \(\tilde{s}\) satisfying \(r+1<\tilde{s} < 1\), apply the same method to \(\tilde{s}\) instead of \(s\). Hence we can assume that \(s<1\). If \(r>-1\), applying \cref{prop:product_rule_hilbert} (ii) with \(t_1 = 0\) and \(t_2= \frac{r+1}{2}\) yields: 
    \begin{align*}
    \norm{u^{2}}_{H^{\frac{r}{2}}} &\lesssim \norm{u}_{L^{2}}\norm{u}_{H^{\frac{r+1}{2}}}\\ 
    &\lesssim \norm{u}_{L^{2}}^{2-\frac{r+1}{s}}\norm{u}_{H^{\frac{s}{2}}}^{\frac{r+1}{s}},
    \end{align*}
    which shows the result with \(\gamma = \frac{r+1}{s}<1\).
    If \(r\leq-1\), we pick \(\tilde{r}\) satisfying \(-1<\tilde{r}<s-1\) and use \(\tilde{r}\) instead of \(r\). 
    \end{proof}}

\subsubsection{Upper bound for \(\Gamma_{\mu}\)}

To be able to estimate the sizes of the functionals \(\mathcal{L}, \mathcal{N}, \tilde{\mathcal{N}}\), we also need an upper bound for \(\Gamma_{\mu}, \tilde{\Gamma}_{\mu}\). In \cite{arnesen2016}, the author shows that \(\Gamma_{\mu}<0\). However, as remarked in a later version of that paper, that estimate only holds for large solutions or if the symbol of the dispersive operator is homogeneous. Fortunately, only \(\Gamma_{\mu}<\mu\) is needed to obtain the necessary estimates in the rest of the proof. We show this bound using a long-wave ansatz for \(\tilde{\Gamma}_{\mu}\), which of course implies \(\Gamma_{\mu}<\mu\) since that is a special case with \(r=0\). 

\begin{lemma}
    The infimum satisfies \(\tilde{\Gamma}_{\mu}<\mu\). \label{lem:infimum_gamma}
\end{lemma}
\begin{proof}
For \(\phi\in \mathcal{S}(\mathbb{R})\) satisfying \(\mathcal{Q}(\phi) = \mu\) and \(0<\theta<1\), let \(\phi_{\theta}(x) = \sqrt{\theta}\phi(\theta x)\).  Then \(\mathcal{Q}(\phi_{\theta}) = \mu\) and by properties of the Fourier transform,
\begin{equation*}
        \mathcal{L}(\phi_{\theta}) = \mu +\frac{1}{2}\int_{\mathbb{R}} (\jpb{\xi}^s-1) \abs{\widehat{\phi_{\theta}}(\xi)}^2\,d\xi= \mu + \frac{1}{2}\int_{\mathbb{R}} (\jpb{\theta\xi}^s-1) \abs{\widehat{\phi}(\xi)}^2\,d\xi.
\end{equation*}
If \(s<2\), then \(\jpb{\theta\xi}^s-1 \lesssim \abs{\theta\xi}^2\), while if \(s\geq 2\), then \(\jpb{\theta\xi}^s-1 \lesssim \abs{\theta\xi}^2 + \abs{\theta\xi}^s\). Since \(0<\theta<1\) and \(\phi\in \mathcal{S}(\mathbb{R})\), this implies that there is a constant \(C_1>0\) such that
\begin{equation*}
    \mathcal{L}(\phi_\theta)\leq \mu + C_1\theta^2.
\end{equation*}
Noting that \(\jpb{\theta\xi}^r\geq \min(\jpb{\xi}^r, 1)\) for all \(\xi\in \mathbb{R}\), we can similarly conclude that
\begin{equation*}
        \tilde{\mathcal{N}}(\phi_{\theta}) = \frac{1}{4} \int_{\mathbb{R}} \jpb{\xi}^r \abs{\widehat{\phi_{\theta}^2}(\xi)}^2\,d\xi = \frac{1}{4}\theta\int_{\mathbb{R}} \jpb{\theta\xi}^r\abs{\widehat{\phi^2}(\xi)}^2\,d\xi \geq C_2\theta
\end{equation*}
for some \(C_2>0\). Thus
\begin{equation*}
    \mathcal{E}(\phi_{\theta})\leq \mu + C_1 \theta^2 - C_2 \theta,
\end{equation*}
and picking \(\theta>0\) small enough ensures that 
\begin{equation*}
    \Gamma_{\mu}\leq \mathcal{E}(\phi_{\theta})<\mu.
\end{equation*} 
\end{proof}

\subsection{Concentration--compactness}
Although \(\Gamma_{\mu}, \tilde{\Gamma}_{\mu}>\infty\), minimizing sequences of \(\Gamma_{\mu}, \tilde{\Gamma}_{\mu}\) do not necessarily converge as \(\mathbb{R}\) is not compact. To overcome this, we use Lions' concentration--compactness principle \cite{lions1984}. Informally, any bounded sequence will either \emph{vanish} (the mass spreads out), \emph{dichotomize} (the mass separates in space) or \emph{concentrate} (most of the mass remains in a bounded domain). The concentration--compactness principle is stated in a form suitable for our purpose in \cite[Lemma 2.6]{arnesen2016}, but we explain what needs to be shown in the next sections. We will apply the result to the sequence \(\{\frac{1}{2}u_n^2\}_{n\in \mathbb{N}}\), where \(\{u_n\}_{n\in \mathbb{N}}\) is a minimizing sequence for \(\Gamma_{\mu}\) or \(\tilde{\Gamma}_{\mu}\). The strategy will be to show that any minimizing sequence will have a subsequence for which we can exclude both vanishing and dichotomy. Then we show that concentration leads to convergence.

\subsubsection{Excluding vanishing}
By definition, the sequence \(\{\frac{1}{2}u_n^2\}_{n\in \mathbb{N}}\) \emph{vanishes} if for all \({r<\infty}\)
\begin{equation*}
    \lim_{n\to\infty}\sup_{y\in\mathbb{R}}\int_{y-r}^{y+r}u_n^2\,dx = 0.
\end{equation*}
It turns out that if \(\norm{u}_{H^{\frac{s}{2}}}<\infty\) and \(\norm{u}_{L^q}\geq\delta>0\) for some suitable \(q\), then \(\{\frac{1}{2}u_n^2\}_{n\in \mathbb{N}}\) cannot vanish, due to the following result.
\begin{lemma}[\!\!{\cite[Lemma 3.4]{arnesen2016}}]
    Let \(v\in H^{\frac{s}{2}}(\mathbb{R})\) and suppose \(q\) satisfies \(q>2\) if \(s\geq 1\) and \(2<q<\frac{2}{1-s}\) if \(s<1\). Given \(\delta>0\), suppose that \({\norm{v}_{H^{\frac{s}{2}}}^{-1}, \norm{v}_{L^q}\geq \delta}\). Then there is \(\varepsilon>0\) such that
    \begin{equation}
        \sup_{j\in\mathbb{Z}}\int_{j-2}^{j+2}\abs{v}^q\,dx \geq \varepsilon.\label{eq:vanish}
    \end{equation}
    \label{lem:vanish}
\end{lemma}
In \cite{arnesen2016} it is then shown that \cref{eq:vanish} and \(\norm{v}_{H^{\frac{s}{2}}}^{-1}>\delta\) implies 
\begin{equation*}
    \sup_{j\in\mathbb{Z}}\int_{j-2}^{j+2}\abs{v}^2\,dx \geq \varepsilon,
\end{equation*}
which of course excludes vanishing when applied to \(u_n\).

In the classical case, \(\norm{u_n}_{H^{\frac{s}{2}}}<\infty\) follows easily from \(\Gamma_{\mu}<\mu\) and \cref{eq:L4_bound_upper} since
\begin{equation*}
    \norm{u_n}_{H^{\frac{s}{2}}}^2\eqsim \mathcal{L}(u_n) = \mathcal{\mathcal{E}}(u_n) + \mathcal{\mathcal{N}}(u_n) < \mu + \mu^{1-\gamma}\norm{u_n}_{H^{\frac{s}{2}}}^{2\gamma}, \quad \gamma<1.
\end{equation*}
Using {only that} \(\mathcal{L}(u)\geq \mu\), \(\Gamma_{\mu}< \mu\) one can also easily deduce that there is a \(\delta>0\) such that \(\mathcal{\mathcal{N}}(u_n) \geq \delta\), see e.g. \cite[Lemma 4.2]{arnesen2016}, passing to a subsequence if necessary. In the classical case, this is all that is needed to apply \cref{lem:vanish}, as {\(\norm{u_n}_{L^4}^4 \eqsim \mathcal{N}(u_n)\geq \delta \)} and \(q=4\) satisfies the assumptions (recall that \(s>1\) in the classical case).

{
For the modified equation \cref{eq:solitary_wave},  \(\norm{u_n}_{H^{\frac{s}{2}}}<\infty\) and \(\tilde{\mathcal{N}}\geq\delta\) hold by the same argument, since the arguments only relied on bounds on the functionals \(\mathcal{E}, \mathcal{L}, \mathcal{N}\) and \(\Gamma_{\mu}\). The same bounds hold for the modified versions, so one may simply replace \(\mathcal{E}, \mathcal{N}, \Gamma_{\mu}\) with \(\tilde{\mathcal{E}}, \tilde{\mathcal{N}}, \tilde{\Gamma}_{\mu}\) and obtain the same results.

When attempting to go from a lower bound on \(\tilde{\mathcal{N}}(u_{n})\) to a lower bound on \(\left\|u_{n} \right\|_{L^{q}}\),} the introduction of the nonlocal nonlinearity gives rise to two new problems. The first is that the estimate \(\norm{u_n}_{L^4}^4\gtrsim \tilde{\mathcal{N}}(u_n)\)
does not hold in general. The second, and bigger problem is that \(q=4\) only satisfies the assumptions of \cref{lem:vanish} when \(s>\frac{1}{2}\), while a sufficiently negative \(r\) allows the minimization problem to be well-defined for \(s>0\). {To overcome this, we need the following lemma.} 

{
\begin{lemma}
    Let \(u\in H^{\frac{s}{2}}(\mathbb{R})\). When \(0<r<s-1\),
    \begin{equation*}
        \norm{u^2}_{H^{\frac{r}{2}}}\leq \norm{u}_{L^4}^{2(1-\frac{r}{s})}\norm{u}_{H^{\frac{s}{2}}}^{\frac{2r}{s}}.
    \end{equation*}
    
    Furthermore, when \(s<1\), there is a \(q\) satisfying \(2<q<{\frac{2}{1-\frac{s}{2}}}\) such that 
    \begin{equation}
        \norm{u^2}_{H^{\frac{r}{2}}}\lesssim \norm{u}_{L^q}\norm{u}_{H^{\frac{s}{2}}}.\label{eq:lemq2}
    \end{equation}\label{lem:q}
\end{lemma}
\begin{proof}
    We show the estimate when \(0<r<s-1\) first. Interpolating, we find that
    \begin{equation*}
        \norm{u^2}_{H^{\frac{r}{2}}}\leq \norm{u^2}_{L^2}^{1-\frac{r}{s}}\norm{u^2}_{H^{\frac{s}{2}}}^{\frac{r}{s}}.
    \end{equation*}
    Since \(s>1\), \(H^{\frac{s}{2}}(\mathbb{R})\) is a Banach algebra. Hence \(
        \norm{u^2}_{H^{\frac{s}{2}}} \lesssim \norm{u}_{H^{\frac{s}{2}}}^2\),
    and we conclude that
    \begin{equation*}
        \norm{u^2}_{H^{\frac{r}{2}}}\leq \norm{u^2}_{L^2}^{1-\frac{r}{s}}\norm{u}_{H^{\frac{s}{2}}}^{\frac{2r}{s}} = \norm{u}_{L^4}^{2-\frac{2r}{s}}\norm{u}_{H^{\frac{s}{2}}}^{\frac{2r}{s}}.
    \end{equation*}

    The estimate when \(s<1\) is shown by a direct calculation, by appealing to Sobolev embedding and \cref{prop:product_rule_general}.
    If \(r>\frac{s}{2}-1\), set \(\tilde{r} = r\). Otherwise, pick \(\tilde{r}\in (\frac{s}{2}-1, s-1)\). Let 
    \begin{equation*}
        p = \frac{2}{1-\tilde{r}} \quad \text{and} \quad q = \frac{2}{\frac{s}{2}-\tilde{r}}.
    \end{equation*} 
    Since \(s<1<2\) and \(\tilde{r}<s-1\) the denominater \(\frac{s}{2}- \tilde{r}>0\). Furthermore, \(\tilde{r}>\frac{s}{2}-1\) guarantees that \(q>2\). Clearly also \(q<{\frac{2}{1-\frac{s}{2}}}\). By Sobolev embedding, valid since \(p< \frac{2}{2-\frac{s}{2}} < 2\) and \(\frac{\tilde{r}}{2} - \frac{1}{2} = - \frac{1}{p}\), 
    \begin{equation*}
        \norm{u^2}_{H^{\frac{r}{2}}}\lesssim \norm{u^2}_{H^{\frac{\tilde{r}}{2}}} \lesssim \norm{u^2}_{L^p}. 
    \end{equation*}
    {Now applying \cref{prop:product_rule_general} with \(t = \frac{s}{2}\) and \(p, q\) as above gives \cref{eq:lemq2}.}
\end{proof}}

{
When \(s\geq 1\), \(q=4\) satisfies the assumptions of \cref{lem:vanish}. In that case, we can bound \(\left\|u_{n} \right\|_{L^4}\) from below in terms of \(\tilde{\mathcal{N}}(u_{n})\) using the first part of \cref{lem:q} when \(r>0\) or the simpler estimate \(\norm{u^2}_{H^{\frac{r}{2}}}^2 \leq \norm{u_n}_{L^4}\) when \(r\leq 0\). When \(s<1\), we can instead use the second part of \cref{lem:q} to find a lower bound for \(\left\|u_{n} \right\|_{L^{q}}, 2<q<\frac{2}{1-\frac{s}{2}}<\frac{2}{1- s}\). In any case, we can conclude that there is a \(\delta>0\) and a suitable \(q\) such that \(\norm{u_n}_{L^q}\geq\delta\) when \(\tilde{\mathcal{N}}(u_n)\geq \tilde{\delta}>0\) and use \cref{lem:vanish} to conclude that vanishing does not occur.
}

\subsubsection{Excluding dichotomy}
The map \(\mu\mapsto \Gamma_{\mu}\) is strictly subadditive:
\begin{equation}
    \Gamma_{\mu_1 + \mu_2} <\Gamma_{\mu_1} + \Gamma_{\mu_2}.\label{eq:subadditivity}
\end{equation}
Subadditivity follows from subhomogeneity, \(\Gamma_{t\mu} <t\Gamma_{\mu}, t>1\) \cite[Lemma 3.3]{arnesen2016}, which in turn follows directly from a scaling argument since \(\mathcal{L}(u), \mathcal{N}(u)\) are both homogeneous and strictly greater than zero \cite[Lemma 4.3]{arnesen2016}. For \(\tilde{\Gamma}_{\mu}\), nothing changes. 

The idea is now to show that dichotomy contradicts \cref{eq:subadditivity}. Here, we follow \cite{maehlen2020}, where the assumptions on the linear, dispersive operator allow for a considerably simpler proof. If \(\{u_n^2\}_{n\in\mathbb{N}}\) dichotomizes, then it is possible to find sequences \(u_n^{(1)}, u_n^{(2)}\) with \({(u_n^{(1)})^2 + (u_n^{(2)})^2 = u_n^2}\) such that
\begin{equation}
    \mathcal{Q}(u_n^{(1)})\to\lambda,\quad \mathcal{Q}(u_n^{(2)})\to(\mu-\lambda), \quad \int_{\supp{u_n^{(1)}}\cap \,\supp{u_n^{(2)}}} u_n^2\,dx \to 0 \label{eq:commutator_L2norm}
\end{equation}
for some \(\lambda\in (0,\mu)\) as \(n\to \infty\) \cite[Proposition 7]{maehlen2020}. In particular, one can { pick two smooth functions \(\phi, \psi \colon \mathbb{R}\to [0,1]\)} with \({\phi(x)=1}\) when \(\abs{x}\leq 1\) and \(\phi(x)=0\) when \(\abs{x}\geq 2\) {and \(\phi^2 + \psi^2 = 1\)}, \sloppy{and find sequences \({\{x_n\}_{n\in \mathbb{N}}, \{R_n\}_{n\in \mathbb{N}}\subset \mathbb{R}}, {R_n \to \infty}\) such that \cref{eq:commutator_L2norm} holds for} \(u_n^{(1)} = \phi_n u_n\) and \(u_n^{(2)}  {= \psi u_{n}} =  \sqrt{1-\phi_n^2}u_n\), where \(\phi_n(x) = \phi(\frac{x-x_n}{R_n})\). We can assume without loss of generality that \(x_n = 0\) for all \(n\). 
Showing that \(\mathcal{E}(u_n)\) eventually also decomposes as  \(u_n^{(1)}, u_n^{(2)}\) separate in space, 
\begin{equation}
    \lim_{n\to \infty} (\mathcal{E}(u_n)) =  \lim_{n\to \infty} (\mathcal{E}(u_n^{(1)}) + \mathcal{E}(u_n^{(2)})),\label{eq:E_decompose}
\end{equation}
gives the desired contradiction, since then
\begin{equation}
    \Gamma_{\mu}= \liminf_{n\to\infty}\mathcal{E}(u_n) \geq \liminf_{n\to\infty}\mathcal{E}(u_n^{1})+\liminf_{n\to\infty}\mathcal{E}(u_n^{2})= \Gamma_{\bar{\mu}} + \Gamma_{(\mu-\bar{\mu})}.\label{eq:contradiction}
\end{equation}
In the classical case, \(\mathcal{N}(u_n)\) is local and the decomposition for this part is automatic,
\begin{equation}
    \mathcal{N}(u_n) \to \mathcal{N}(u_n^{(1)}) + \mathcal{N}(u_n^{(2)}) \quad\text{as}\quad n\to \infty.\label{eq:N_decompose}
\end{equation}
Since \(\Lambda^s\) is nonlocal, the same decomposition is not automatically true for \(\mathcal{L}(u_n)\). However, it holds because of the regularity of the symbol \(\jpb{\xi}^{s}\) (here, the symbol is smooth, but continuously differentiable suffices). {The key is the following commutator estimate.}
\begin{lemma}[\!\!{\cite[Lemma 6.1]{maehlen2020}}]
    Let \(u, v\in H^{\frac{s}{2}}(\mathbb{R})\) and let \(\rho\in \mathcal{S}(\mathbb{R})\) be a non-negative Schwartz function. Define \(\rho_R(x) = \rho(x/R)\). Then
    \begin{equation*}
    |\int_{\mathbb{R}}v (\rho_R \Lambda^s u - \Lambda^s(\rho_R u))\,dx| \to 0
    \end{equation*}
    as \(R\to \infty\). 
    \label{lem:commutator_L}
\end{lemma}
With this estimate, it is an easy matter, see e.g. \cite[Proposition 7]{maehlen2020}, to show that
\begin{equation}
    \mathcal{L}(u_n) \to \mathcal{L}(u_n^{(1)}) + \mathcal{L}(u_n^{(2)}) \quad\text{as}\quad n\to \infty.\label{eq:L_decompose}
\end{equation}

Since \(\tilde{\mathcal{N}}\) contains a nonlocal operator, \cref{eq:N_decompose} but for \(\tilde{\mathcal{N}}\) is not automatic. {We will show it using a similar strategy as for \(\mathcal{L}\), beginning with a version of \cref{lem:commutator_L} with \(\Lambda^r\) and \(u^2\).}  
Since \(r\) can be negative and \(u^2, v^2\) is not necessarily in \(H^{\frac{s}{2}}(\mathbb{R})\) or even \(L^2(\mathbb{R})\), the proof is more involved than for \cref{lem:commutator_L}.
{
\begin{lemma}
    Let \(u, v \in H^{\frac{s}{2}}(\mathbb{R})\) and let \(\rho\in \mathcal{S}(\mathbb{R})\) be a non-negative Schwartz function. Define \(\rho_R(x) = \rho(x/R)\). Then
    \begin{equation*}
    |\int_{\mathbb{R}}v^{2} (\rho_R \Lambda^r u^{2} - \Lambda^r(\rho_R u^{2}))\,dx| \to 0 \quad \text{as} \quad R\to \infty.
    \end{equation*} 
    \label{lem:commutator_N}
\end{lemma}
\begin{proof}
    By \cref{lem:N_upper}, \(u^2 \in H^{\frac{r}{2}}(\mathbb{R})\). Using Plancherel's and Fubini's theorems, 
    \begin{equation}
        \begin{split}
        &\abs{\int_{\mathbb{R}}v^{2} (\rho_R \Lambda^r u^{2} - \Lambda^r(\rho_R u^{2}))\,dx} \\
        &= \abs{\int_{\mathbb{R}} \overline{\widehat{v^2}}(\xi)((\widehat{\rho_R}\ast \jpb{\cdot}^r\widehat{u^2})(\xi)- \jpb{\xi}^{r} (\widehat{\rho_R}\ast \widehat{u^2})(\xi))\,d\xi}\\
        &\leq \int_{\mathbb{R}} \abs{\widehat{\rho_R}(t)}\int_{\mathbb{R}}\abs{\overline{\widehat{v^2}}(\xi)} \abs{\widehat{u^2}(\xi-t)}\underbrace{\abs{\jpb{\xi-t}^r - \jpb{\xi}^r}}_{D}\,d\xi\,dt \coloneqq I.
        \end{split}
        \label{eq:commutator_start}
    \end{equation}
    We split the domain of integration of the outer integral into \(\abs{t}\geq R^{-\frac{1}{2}}\) and \({\abs{t}< R^{-\frac{1}{2}}}\) and show that each part approaches zero as \(R\to \infty\). For the first part, \(\abs{t}\geq R^{-\frac{1}{2}}\), we apply the triangle inequality for a coarse upper bound for the difference \(D\): \({\abs{\jpb{\xi-t}^r - \jpb{\xi}}\leq \jpb{\xi-t}^r + \jpb{\xi}^r}.\)
    Observe that \(\jpb{\xi+t}^{\frac{r}{2}}\lesssim \jpb{\xi}^{\frac{r}{2}}\jpb{t}^{\frac{\abs{r}}{2}}\),
    so that \(\norm{u(\cdot-t)}_{H^{\frac{r}{2}}}\lesssim \jpb{t}^{\frac{\abs{r}}{2}}\norm{u}_{H^{\frac{r}{2}}}\).
    Combining these estimates, we find that
    \begin{equation}
        \begin{split}
            &\int_{\abs{t}\geq R^{-\frac{1}{2}}} \abs{\widehat{\rho_R}(t)}\int_{\mathbb{R}} \abs{\overline{\widehat{v^2}}(\xi)} \abs{\widehat{u^2}(\xi-t)}\abs{\jpb{\xi-t}^r - \jpb{\xi}^r}\,d\xi\,dt\\
        &\,\,\leq \int_{\abs{t}\geq R^{-\frac{1}{2}}} \abs{\widehat{\rho_R}(t)}\int_{\mathbb{R}}\abs{\overline{\widehat{v^2}}(\xi)\jpb{\xi-t}^{\frac{r}{2}}} \abs{\widehat{u^2}(\xi-t)\jpb{\xi-t}^{\frac{r}{2}}}+\abs{\overline{\widehat{v^2}}(\xi)\jpb{\xi}^{\frac{r}{2}}} \abs{\widehat{u^2}(\xi-t)\jpb{\xi}^{\frac{r}{2}}}\,d\xi\,dt\\
        &\,\,\lesssim \norm{v^2}_{H^{\frac{r}{2}}}\norm{u^2}_{H^{\frac{r}{2}}}\int_{\abs{t}\geq R^{-\frac{1}{2}}}\jpb{t}^{\frac{\abs{r}}{2}} R\abs{\widehat{\rho}(Rt)}\,dt= \norm{v^2}_{H^{\frac{r}{2}}}\norm{u^2}_{H^{\frac{r}{2}}}\int_{\abs{t}\geq R^{\frac{1}{2}}}\jpb{\frac{t}{R}}^{\frac{\abs{r}}{2}} \abs{\widehat{\rho}(t)}\,dt.
        \end{split}
        \label{eq:est_large_t}
    \end{equation}
    Since \(\rho\in \mathcal{S}(\mathbb{R})\), this expression approaches zero as \(R\to \infty\). For the second part, \(\abs{t}<R^{-\frac{1}{2}}\), we can assume that \(\abs{t}\leq 1\) since \(R^{-\frac{1}{2}}\to 0\) as \(R\to \infty\). With the mean value theorem, we find a refined upper bound for \(D\):
    \begin{equation*}
        \abs{\jpb{\xi-t}^r - \jpb{\xi}^r} \leq \sup_{\theta \in (\min(\xi-t, \xi), \max(\xi-t, t))} \abs{t} \jpb{\theta}^{r-1} \lesssim \abs{t}\jpb{\xi}^{r-1}.
    \end{equation*} 
    Thus
    \begin{equation}
        \begin{split}
        &\int_{\abs{t}< R^{-\frac{1}{2}}} \abs{\widehat{\rho_R}(t)}\int_{\mathbb{R}} \abs{\overline{\widehat{v^2}}(\xi)} \abs{\widehat{u^2}(\xi-t)}\abs{\jpb{\xi-t}^r - \jpb{\xi}^r}\,d\xi\,dt\\
        &\quad\leq \int_{\abs{t}< R^{-\frac{1}{2}}} R\abs{\widehat{\rho}(Rt)}\abs{t}\int_{\mathbb{R}} \abs{\overline{\widehat{v^2}}(\xi)\jpb{\xi}^{\frac{r-1}{2}}} \abs{\widehat{u^2}(\xi-t)\jpb{\xi}^{\frac{r-1}{2}}}\,d\xi\,dt\\
        &\quad\lesssim \norm{v^2}_{H^{\frac{r-1}{2}}}\norm{u^2}_{H^{\frac{r-1}{2}}}\frac{1}{R}\int_{\mathbb{R}} \abs{\widehat{\rho}(t)}\abs{t}\,dt
        \end{split}
        \label{eq:est_small_t}
    \end{equation}
    which also approaches zero as \(R\to \infty\).
    Combining \cref{eq:est_large_t,eq:est_small_t} with \cref{eq:commutator_start}, clearly \(I\to 0\) as \(R\to \infty\).
\end{proof}}

{ With this lemma, we can show that \(\mathcal{N}(u_{n})\) decomposes as \(n\to\infty\), that is, \cref{eq:N_decompose} holds for \(\tilde{\mathcal{N}}\).
} Since \({(u_n^{(1)})^2 + (u_n^{(2)})^2 = u_n^2}\), 
\begin{equation*}
    \begin{split}
        \tilde{\mathcal{N}}(u_n) &= \tilde{\mathcal{N}}(u_n^{(1)}) + \tilde{\mathcal{N}}(u_n^{(2)}) + \frac{1}{2}\int_{\mathbb{R}} (1-\phi_n^2) u_n^2 \Lambda^r \phi_n^2u_n^2\,dx,
    \end{split}
\end{equation*}
and we will show that the last term approaches zero as \(n\to \infty\). First, observe that {\cref{lem:commutator_N}} still holds if we replace {\(\rho_{R}\)} with {\(1-\rho_{R}\)} since
\begin{equation*}
    \int_{\mathbb{R}}v^{2} ((1-\rho_R) \Lambda^r u^{2} - \Lambda^r((1-\rho_R) u^{2}))\,dx = -\!\!\int_{\mathbb{R}}v^{2} (\rho_R \Lambda^r u^{2} - \Lambda^r(\rho_R u^{2}))\,dx.
\end{equation*}
Applying this and {\cref{lem:commutator_N}} with \(\phi^2\) as \(\rho\) and \(\sqrt{1-\phi^2}\) as \(1-\rho\) respectively, we find that 
\begin{equation*}
    \lim_{n\to\infty}\int_{\mathbb{R}} (1-\phi_n^2) u_n^2 \Lambda^r \phi_n^2u_n^2\,dx= \lim_{n\to\infty}\tilde{\mathcal{N}}(\phi_n \sqrt{1-\phi_n^2}u_n).
\end{equation*}
By \cref{eq:N_upper}, for some \(\gamma<2\),
\begin{equation*}
    \tilde{\mathcal{N}}(\phi_n \sqrt{1-\phi_n^2}u_n) \lesssim \norm{\phi_n \sqrt{1-\phi_n^2}u_n}_{L^2}^{4-2\gamma}\norm{\phi_n \sqrt{1-\phi_n^2}u_n}_{H^{\frac{s}{2}}} \to 0
\end{equation*}
where the limit as \(n\to \infty\) follows from \cref{eq:commutator_L2norm}. {Thus we conclude that
\begin{equation*}
    \tilde{\mathcal{N}}(u_n) \to \tilde{\mathcal{N}}(u_n^{(1)}) + \tilde{\mathcal{N}}(u_n^{(2)}) \quad\text{as}\quad n\to \infty.
\end{equation*}
Combining this with \cref{eq:L_decompose}, we see that \cref{eq:E_decompose} and \cref{eq:contradiction} must hold for \(\tilde{\mathcal{E}}, \tilde{\Gamma}_{\mu}\) and we have the desired contradiction.}

\subsubsection{Convergence from concentration}
Having excluded vanishing and dichotomy, we conclude that a subsequence of \(\{\frac{1}{2}u_n^2\}_{n\in\mathbb{N}}\) \emph{concentrates}. We can then show that it converges in \(H^{\frac{s}{2}}(\mathbb{R})\) to a minimizer of \(\Gamma_{\mu}\). Since a minimizing sequence of \(\Gamma_{\mu}\) is bounded in \(H^{\frac{s}{2}}(\mathbb{R})\), it converges weakly -- up to a subsequence -- in \(H^{\frac{s}{2}}(\mathbb{R})\) to some \(\omega \in H^{\frac{s}{2}}(\mathbb{R})\). 
From this, one can use standard arguments to show that \(\{u_n\}_{n\in\mathbb{R}}\) converges -- up to subsequences and translations -- in \(L^2(\mathbb{R})\) to \(\omega\), see e.g. \cite[Lemma 4.6]{arnesen2016} or \cite[Proposition 8]{maehlen2020}. In the classical case, \cref{eq:L4_bound_upper} tells us that \(u_n \to \omega\) in \(L^4(\mathbb{R})\) as well so that \(\mathcal{N}(u_n)\to\mathcal{N}(\omega)\). Using this and the weak {lower} semi-continuity of the \(H^{\frac{s}{2}}(\mathbb{R})\)-norm,
\begin{equation*}
    \Gamma_{\mu}\leq \mathcal{L}(\omega) - \mathcal{N}(\omega) \leq \liminf_{n\to\infty} (\mathcal{L}(u_n) - \mathcal{N}(u_n)) = \Gamma_{\mu}.
\end{equation*}
Thus \(\mathcal{E}(u_n)\to \mathcal{E}(\omega)\). Since \(\mathcal{N}(u_n)\to\mathcal{N}(\omega)\), then \(\mathcal{L}(u_n)\to \mathcal{L}(\omega)\), which together with the weak \(H^{\frac{s}{2}}\)-convergence implies that \(u_n\to \omega\) in \(H^{\frac{s}{2}}(\mathbb{R})\).

For \(\tilde{\mathcal{N}}(u)\), \cref{eq:L4_bound_upper} does not necessarily hold, so we need to show that \(\tilde{\mathcal{N}}(u_n) \to \tilde{\mathcal{N}}(\omega)\) if \(u_n\to \omega\) in \(L^2(\mathbb{R})\). Since \({\tilde{\mathcal{N}}(u) \eqsim \norm{u^2}_{H^{\frac{r}{2}}}^{2}}\), it suffices to show that \(u_n^2 \to \omega^2\) in \(H^{\frac{r}{2}}(\mathbb{R})\). 
{
\begin{lemma}
    Let \(u, v \in H^{\frac{s}{2}}\). Then there exists \(0<\alpha\leq 1\) such that
    \[
    \left\|u^{2}-v^{2} \right\|_{H^{\frac{r}{2}}} \lesssim \left\|u-v \right\|_{L^{2}}^{\alpha} \left\|u-v \right\|_{H^{\frac{s}{2}}}^{1-\alpha} \left\|u+v \right\|_{H^{\frac{s}{2}}}
    .\] \label{lem:conv}
\end{lemma}
\begin{proof}
 For \(\bar{u},\bar{v}\in H^{\frac{s}{2}}( \mathbb{R})\), there is always \(0<\alpha\leq 1\) such that
    \begin{equation*}
        \norm{\bar{u}\bar{v}}_{H^{\frac{r}{2}}}\lesssim \norm{\bar{u}}_{L^2}^{\alpha}\norm{\bar{u}}_{H^{\frac{s}{2}}}^{1-\alpha}\norm{\bar{v}}_{H^{\frac{s}{2}}}.
    \end{equation*}
    {
    This can be seen by applying \cref{prop:product_rule_hilbert} to \(\bar{u}\bar{v}\) in the same manner as in the proof of \cref{lem:N_upper}, using the lowest regularity \(t_1\) for \(\bar{u}\) and then interpolating only that part. }

    Noting that \(u^2 - v^2= (u - v)(u + v)\), we can apply the above equation to \(\bar{u} = u-v, \bar{v} = u+v\) to obtain the result.
\end{proof}

 Since \(u_{n}, \omega \in H^{\frac{s}{2}}( \mathbb{R})\), 
 \cref{lem:conv} and the already established convergence in \(L^2(\mathbb{R})\) implies that \(u_{n}^2\) converges strongly to \(\omega^2\) in \(H^{\frac{r}{2}}( \mathbb{R})\).
 }{
 Since a minimizing sequence for \(\tilde{\Gamma}_{\mu}\) satisfies the same properties as a minimizing sequence for \(\Gamma_{\mu}\), the other arguments are unchanged. This concludes the proof of existence of solitary-waves solutions to \cref{eq:solitary_wave} for any \(\mu>0\).}

\subsection{Wave speed and regularity { for \(\mu>0\)}} \label{sec:wave_speed_reg}
We estimate the wave speed and regularity for minimizers of \(\tilde{\Gamma}_{\mu}\) directly, based on similar results in \cite{maehlen2020,arnesen2016}. We show that the wave speed \(\nu\) is subcritical by multiplying \cref{eq:solitary_wave} with the solution \(u\){, integrating} and reshuffling: 
\begin{equation}
    \nu = {\frac{ \int_{ \mathbb{R}} u\Lambda^su \,dx - \int_{ \mathbb{R}} u^2 \Lambda^r u^2 \, dx}{ \int_{ \mathbb{R}} u^2 \, dx}} = \frac{\tilde{\mathcal{E}}(u)- \tilde{\mathcal{N}}(u)}{\mathcal{Q}(u)} < \frac{\tilde{\mathcal{E}}(u)}{\mathcal{Q}(u)} <1. \label{eq:nu}
\end{equation}
Here, we used that \(\tilde{\mathcal{N}}(u)>0\) and \(\tilde{\Gamma}_{\mu} = \tilde{\mathcal{E}}(u) <\mu\). 

A consequence of this is that the operator \((\Lambda^s-\nu)\colon H^{t}(\mathbb{R})\to H^{t - s}(\mathbb{R})\) is invertible and we can write \cref{eq:solitary_wave} as 
\begin{equation}
    u = (\Lambda^s-\nu)^{-1}(u\Lambda^r u^2).\label{eq:bootstrap}
\end{equation}
We only know that the solution \(u\) is in \(H^{\frac{s}{2}}(\mathbb{R})\), but using the formulation above we can show that the solution inherits regularity from the equation itself. {We need the following lemma.}  

{
\begin{lemma}
    Let \(t\geq \frac{s}{2}\) and suppose \(u\in H^{t}(\mathbb{R})\). Then \({u\Lambda^r u^2 \in H^{\min(t, t-r)}(\mathbb{R})}\). \label{lem:regularity}
\end{lemma}
\begin{proof}
    Suppose first that \(t>\frac{1}{2}\). Then \(H^t(\mathbb{R})\) is a Banach algebra and \(u^2\in H^{t}(\mathbb{R})\). This implies in turn that \(\Lambda^r u^2\in H^{t-r}(\mathbb{R})\). If \(r\leq 0\), then \(t-r\geq t\) and we conclude that \(u\Lambda^r u^2\in H^{t}(\mathbb{R})\). If on the other hand \(r>0\), we use \cref{prop:product_rule_hilbert} with \(t_1 = t-r, t_2 = t\) to conclude that \(u\Lambda^r u^2 \in H^{t-r}(\mathbb{R})\). 

    If \(t<\frac{1}{2}\), then \cref{prop:product_rule_hilbert} with \(t_1 = t_2 = t\) implies that \(u^2 \in H^{2t-\frac{1}{2}}(\mathbb{R})\). Then \(\Lambda^r u^2 \in H^{2t-\frac{1}{2}-r}(\mathbb{R})\). Since by assumption \(s\leq t\), then \(2t - \frac{1}{2} - r> 2t - s+1-\frac{1}{2} \geq \frac{1}{2}>t\) and \(u\Lambda^r u^2 \in H^{t}(\mathbb{R})\). Finally, if \(t= \frac{1}{2}\), pick \(\frac{s}{2}<\tilde{t}<t\) and do as above. 
\end{proof}
}

Now assume that \({u\in H^{t}(\mathbb{R})}, {t\geq \frac{s}{2}}\). {Then by the above lemma,} \(u\Lambda^r u^2 \in H^{\min(t, t-r)}(\mathbb{R})\) which implies that \(u\in H^{\min(t, t-r) + s}(\mathbb{R})\) due to \cref{eq:bootstrap}. Since \(s\) is positive and larger than \(r\), this gives an improved regularity estimate for \(u\). We start with \(t_0= \frac{s}{2}\) and find \(u\in H^{t_1}(\mathbb{R}), t_1 = \min(t_0, t_0-r) + s\), then \(u\in H^{t_2}(\mathbb{R}), t_2 = \min(t_1, t_1-r) + s\) and so on. Repeating this procedure gives increasingly better regularity estimates for \(u\in H^{\frac{s}{2}}\). Since we can continue indefinitely, we conclude that \(u\in H^{\infty}(\mathbb{R}) = \cap_{t\in \mathbb{R}}H^{t}(\mathbb{R})\).

{ With this, we have shown the first part of \cref{thm:main} which holds for all \(\mu>0\).} 

{
\subsection{Estimates for norms and wave speed when \(\mu \in (0,\mu_{0})\)}
It remains to show the estimates in \cref{eq:small_solutions_est}, which hold uniformly in \(\mu \in (0,\mu_{0})\) for any fixed \(\mu_{0}>0\). The argument is similar to that in Propositions 4, 9, and 10 in \cite{maehlen2020}. In the following, we fix \(\mu_{0}>0\). 

To show that \(\left\|u \right\|_{H^{\frac{s}{2}}}\eqsim \mu^{\frac{1}{2}}\), we first find a crude upper bound for \(\tilde{\mathcal{N}}(u)\). Using \cref{eq:N_upper} and \cref{lem:infimum_gamma},
\begin{equation*}
    \tilde{\mathcal{N}}(u) \lesssim \mu^{2- \gamma} \left\|u \right\|_{H^{\frac{s}{2}}}^{2\gamma}\leq \mu^{2- \gamma}(\tilde{\mathcal{N}}(u) + \mu)^{\gamma}\lesssim \mu^{2} + \mu^{2-\gamma}\tilde{\mathcal{N}}(u)^{\gamma},
\end{equation*}
for some \(\gamma\in (0,1)\). This implies that
\begin{equation}
    \tilde{\mathcal{N}}(u)\lesssim \mu^2, \quad \text{and} \quad\mu \lesssim \left\|u \right\|_{H^{\frac{s}{2}}}^{2} \eqsim \mathcal{L}(u)\leq \mu + \tilde{\mathcal{N}}(u) \lesssim \mu.\label{eq:Hest}
\end{equation} 
for \(\mu \in (0,\mu_{0})\) with implicit constants depending on \(\mu_{0}\).

We proceed to show the estimate of the wave speed in \cref{eq:small_solutions_est}. To do so, we need an improved estimate for \(\tilde{\mathcal{N}}(u)\) in terms of \(\mu\). As in \cite{maehlen2020}, we find an improved upper bound for the infimum \(\tilde{\Gamma}_{\mu}\) and decompose \(u\) into high- and low-frequency components to estimate \(\tilde{\mathcal{N}}(u)\) precicely. 
\begin{lemma} Let \(\mu_{0}>0\) and \(\mu \in (0, \mu_{0})\). A minimizer \(u\) of \(\tilde{\Gamma}_{\mu}\) satisfies
    \begin{equation*}
        \tilde{\mathcal{N}}\eqsim \mu^3.
    \end{equation*} \label{lem:N_improved}
\end{lemma}
\begin{proof}
    The first step is to show that there is a \(\kappa>0\) such that 
     \begin{equation}
        \tilde{\Gamma}_{\mu} < \mu - \kappa \mu^3 \label{eq:infimum_improved}
    \end{equation}
    Pick \(\phi\in \mathcal{S}( \mathbb{R})\) such that \(\mathcal{Q}(\phi)= 1\), and define \(\phi_{\mu, t}(x) = \sqrt{\mu t}\phi(tx)\). Then \(\mathcal{Q}(\phi_{\mu,t}) = \mu\). In a similar manner as in the proof of \cref{lem:infimum_gamma}, one may show that there are constants \(C_{1}, C_{2}>0\) such that
    \begin{align*}
        \mathcal{L}(\phi_{\mu, t})&\leq \mu + C_{1} t^2 \mu,\\
        \tilde{\mathcal{N}}(\phi_{\mu, t})&\geq C_{2} \mu^{2} t.
    \end{align*}
    Set \(t= C_{3}\mu\), with \(C_{3}<\frac{1}{\mu_{0}}\), such that \(t<1\) for all \(\mu \in (0,\mu_{0})\). Then 
    \begin{equation*}
        \tilde{\mathcal{E}}(\phi_{\mu,t}) \leq \mu - \mu^3(C_{2}C_{3}- C_{1}C_{3}^2) = \mu - \kappa \mu^3,
    \end{equation*}  
    where we can guarantee that \(\kappa>0\) by picking \(C_{3}\) sufficiently small.
    
    The lower bound on \(\tilde{\mathcal{N}}(u)\) follows immediately from \cref{eq:infimum_improved}, 
    \begin{equation*}
        \tilde{\mathcal{N}}(u) = \mathcal{L}(u) - \tilde{\mathcal{E}}(u) \geq \kappa \mu^3.
    \end{equation*}
    To bound \(\tilde{\mathcal{N}}(u)\) from above, let \(\chi\) denote the characteristic function equal to \(1\) on \([-1,1]\) and \(0\) otherwise. Set \(u = u_{1} + u_{2}\), where \(\hat{u}_{1} = \chi \hat{u}, \hat{u}_{2} = \hat{u}\). Then \(\left\|u_{2}\right\|_{H^{\frac{s}{2}}}\eqsim \left\|u_{2} \right\|_{\dot{H}^{\frac{s}{2}}}\) and 
    \begin{equation}
        \mathcal{L}(u) - \mu = \frac{1}{2}\int_{\mathbb{R}} (\left\langle \xi \right\rangle^{s}-1)\left|\hat{u}_1 + \hat{u}_{2} \right|^{2}\,d\xi \eqsim \left\| u_{1}\right\|_{\dot{H}^{1}}^{2} + \left\|u_{2} \right\|_{\dot{H}^{\frac{s}{2}}}^{2},\label{eq:est_hom}
    \end{equation}
    since \(\left\langle \xi \right\rangle^{s} -1\) grows like \(\left|\xi\right|^{s}\) or \(\left|\xi \right|^{2}\) when \(\left|\xi \right|\) is large or small respectively, and \(\hat{u}_1, \hat{u}_2\) have disjoint supports. Now, 
    \begin{equation}
        \tilde{\mathcal{N}}(u) \eqsim \left\|(u_{1} + u_{2})^{2} \right\|_{H^{\frac{r}{2}}}^{2} \lesssim \left\| u_{1}^{2} \right\|_{H^{\frac{r}{2}}}^{2} + \left\|u_{2}^2 \right\|_{H^{\frac{r}{2}}}^{2} + \left\|u_{1}u_{2} \right\|_{H^{\frac{r}{2}}}^{2}, \label{eq:estN}
    \end{equation}
    and we estimate the three terms separately. Using \cref{lem:N_upper,eq:est_hom,eq:Hest}, 
    \begin{equation}
        \left\|u_{2}^2 \right\|_{H^{\frac{r}{2}}}^{2} \lesssim \left\|u_{2} \right\|_{H^{\frac{s}{2}}}^{4} \eqsim \left\|u_{2} \right\|_{\dot{H}^{\frac{s}{2}}}^4 \lesssim (\mathcal{L}(u)-\mu)^{2} \lesssim \tilde{\mathcal{N}}(u)^2 \lesssim \mu^{\frac{3}{2}}\tilde{\mathcal{N}}(u)^{\frac{1}{2}}.\label{eq:est1}
    \end{equation}
    Similarly, 
    \begin{equation}
        \left\|u_{1}u_{2} \right\|_{H^{\frac{r}{2}}}^2 \lesssim \left\|u_{1} \right\|_{H^{\frac{s}{2}}}^{2} \left\|u_{2} \right\|_{\dot{H}^{\frac{s}{2}}}^2 \lesssim \mu (\mathcal{L}(u)-\mu) \lesssim \mu \tilde{\mathcal{N}}(u)\lesssim\mu^{\frac{3}{2}}\tilde{\mathcal{N}}(u)^{\frac{1}{2}}.\label{eq:est2}
    \end{equation}
    Estimating \(\left\|u_{1}^2 \right\|_{H^{\frac{r}{2}}}^2\) requires a little more care. If \(r\leq 0\), then  \(\left\|u_{1}^2 \right\|_{H^{\frac{r}{2}}}^2 \lesssim \left\|u_{1} \right\|_{L^4}^4 \), whereas if \(r>0\), then 
    \begin{equation*}
        \left\|u_{1}^2 \right\|_{H^{\frac{r}{2}}}^2 = \left\|\left\langle \cdot \right\rangle^{\frac{r}{2}} \hat{u}_1\ast \hat{u}_1 \right\|_{L^2}^2 \lesssim  \left\|\Lambda^{\frac{r}{2}}u_{1} \right\|_{L^4}^4, 
    \end{equation*}
    where we used that \(\left\langle \xi \right\rangle \lesssim \left\langle \xi-\eta \right\rangle \left\langle \eta \right\rangle\). Since \(\left\|u_{1} \right\|_{L^4}^4 \lesssim \left\|\Lambda^{\frac{r}{2}}u_{1} \right\|_{L^4}^4\) for \(r>0\), it suffices to estimate the latter. We can use the Gagliardo-Nirenberg-Sobolev interpolation inequality, \cref{eq:Hest} and \cref{eq:est_hom} to find that 
    \begin{equation}
        \left\|\Lambda^{\frac{r}{2}}u_{1} \right\|_{L^4}^4 \lesssim \left\|\Lambda^{\frac{r}{2}}u_{1} \right\|_{L^{2}}^3 \left\|\Lambda^{\frac{r}{2}}u_{1}  \right\|_{\dot{H}^{1}} \lesssim \left\|u_{1} \right\|_{H^{\frac{s}{2}}}^{3} \left\|u_{1} \right\|_{\dot{H}^1} \lesssim \mu^{\frac{3}{2}} (\mathcal{L}-\mu)^{\frac{1}{2}}\lesssim \mu^{\frac{3}{2}} \tilde{\mathcal{N}}(u)^{\frac{1}{2}}.\label{eq:est3}
    \end{equation}
    Combining \cref{eq:estN,eq:est1,eq:est2,eq:est3} we conclude that
    \[
    \tilde{\mathcal{N}}(u) \lesssim \mu^3.
    \] 
\end{proof}
The estimate of the wave speed in \cref{eq:small_solutions_est} is an immediate consequence of \cref{lem:N_improved}. Using \cref{eq:nu}, \(\mathcal{L}(u)\geq \mu, \tilde{\mathcal{E}}(u)<\mu\) and  \(\tilde{\mathcal{N}}(u)\eqsim \mu^3\), 
\begin{equation*}
    1- \nu =  \frac{\mu - \tilde{\mathcal{E}}(u) + \tilde{\mathcal{N}}(u)}{\mu} \geq \frac{\tilde{\mathcal{N}}(u)}{\mu} \eqsim \mu^2  \eqsim\frac{2\tilde{\mathcal{N}}(u)}{\mu} \geq \frac{\mu - \mathcal{L}(u) + 2\tilde{\mathcal{N}}(u)}{\mu} = 1 - \nu.
\end{equation*}

Finally, we want to estimate \(\left\|u \right\|_{L^{\infty}}\), finishing the proof of \cref{eq:small_solutions_est}. This can be done by estimating \(\left\|u \right\|_{L^{\infty}}\lesssim \left\|u \right\|_{H^{1}}\) and then using a similar bootstrap argument as in \cref{sec:wave_speed_reg}. However, we need to ensure that the estimate \(\left\|u \right\|_{H^1} \lesssim \left\|u \right\|_{H^{\frac{s}{2}}}\) is uniform in \(\mu\in(0,\mu_{0})\). 
Increasing powers of \(\mu\) in each iteration does not cause problems as we can always estimate \(\mu\leq \mu_{0}\), but we want to make sure that we do not lose any powers of \(\mu\) when we estimate \(u\) by \(u\Lambda^r u^2\). We used that the operator \((\Lambda^s- \nu)\) was invertible, but the operator norm of \((\Lambda^s - \nu)^{-1}\) will blow up as \(\mu\to 0\). Therefore, we instead rewrite \cref{eq:solitary_wave} as
\begin{equation}
    (\Lambda^s - \nu + 1)u = u\Lambda^ru^2 + u,\label{eq:bootstrap2}
\end{equation}
and note that the operator norm of \((\Lambda^s - \nu +1)^{-1}\colon H^{\alpha}(\mathbb{R})\to H^{\alpha +t}(\mathbb{R})\) is independent of \(\mu\in (0,\mu_{0})\),
\begin{equation*}
    \left\|\left( \Lambda^s- \nu + 1 \right)^{-1} \right\|_{H^{\alpha}\to H^{\alpha + t}}= \sup_{\xi\in\mathbb{R}}\frac{\left\langle \xi \right\rangle^{s} }{\left\langle \xi \right\rangle^{s} - \nu + 1 }\leq  1,
\end{equation*}
since \(1-\nu \eqsim \mu^{2}\) (see also Proposition 10 in \cite{maehlen2020}). \Cref{lem:regularity} clearly also holds with \(u\Lambda^r u^2 + u\) instead of \(u\Lambda^r u^2\), and so the argument is exactly the same as in \cref{sec:wave_speed_reg}, only using \cref{eq:bootstrap2} instead of \cref{eq:bootstrap}. We conclude that 
\begin{equation*}
    \left\|u \right\|_{L^{\infty}}\lesssim \left\|u \right\|_{H^1}\lesssim \left\|u \right\|_{H^{\frac{s}{2}}}\eqsim \mu.
\end{equation*}

}
\section{Some comments on possible generalizations}
\label{sec:general_assumptions}
In \cite{arnesen2016}, \cite{maehlen2020}, existence of solitary waves to \cref{eq:model_general} is proved for a more general operator than \(L=\Lambda^{s}\). Likewise, the modifications due to the nonlocal nonlinearity presented in this paper do not rely heavily on special properties of \(\Lambda^r\). In fact, for the equation
\begin{equation*}
    \partial_t u + \partial_x(Lu - uNu^2) = 0, 
\end{equation*}
the proof of \cref{thm:main} will go through under the following assumptions on \(L\) and \(N\).
\begin{itemize}
    \item[(A1)] The symbol \(m\in C(\mathbb{R})\) of the operator \(L\) is real-valued, positive, even and satisfies the growth bounds
    \begin{align*}
        m(\xi) - m(0) \eqsim |\xi|^s \quad\,\, &\text{for }|\xi|\geq 1,\\
        m(\xi) - m(0) \eqsim |\xi|^{s'} \quad &\text{for }|\xi| <1,
        \end{align*}
    with \(s> 1/2, s'> 1\). Furthermore, we require that \(\xi \mapsto m(\xi)/\jpb{\xi}^s\) be uniformly continuous on \(\mathbb{R}\). 
    \item[(A2)] The symbol \(n\in C^1(\mathbb{R})\) of the operator \(N\) is real-valued, even, and satisfies the growth bounds
    \begin{align*}
        n(\xi) &\eqsim\jpb{\xi}^{r} \quad\quad \text{for }\xi \in \mathbb{R},\,\,\\
        \Big|\frac{\partial n}{\partial \xi}(\xi)\Big|&\lesssim \jpb{\xi}^{r-1} \quad \text{for }\xi\in\mathbb{R}.
    \end{align*}
    where \(r<s-1\).     
\end{itemize}
The assumptions on \(L\) are as in \cite{maehlen2020}. One may observe that \(m(\xi) = \jpb{\xi}^s\) satisfies these assumptions with \(s'=2\). We refer to \cite{maehlen2020} for a detailed discussion and comment only on a few important changes. Since \(m\) can be homogenous, we use \(\norm{u}_{H^{\frac{s}{2}}}^2\eqsim \mathcal{Q}(u) + \mathcal{L}(u)\) instead of \(\norm{u}_{H^{\frac{s}{2}}}^2\eqsim \mathcal{L}(u)\), and the upper bound for the infimum becomes \(\Gamma_{\mu}<m(0)\mu\) instead of \(\Gamma_{\mu}<\mu\). The uniform continuity of \(m(\xi)/\jpb{\xi}^s\) is used to exclude dichotomy. 

For the symbol \(n\), a similar regularity assumption to the one on \(m\) would allow us to prove \cref{lem:commutator_N} and exclude dichotomy. However, while a Fourier multiplier is continuous on \(L^2(\mathbb{R})\) as long as it is uniformly bounded, a stronger assumption is needed to ensure continuity on \(L^{p}(\mathbb{R}), p\neq 2\). There are different criteria that guarantee this; continuous differentiability and the bound on \(\partial n/\partial \xi\) suffices \cite[Theorem 5.2.7]{grafakos2014}. Unlike \(m\), the symbol \(n\) must be inhomogeneous to establish the upper bound on \(\Gamma_{\mu}\) by a long-wave ansatz as is done here. A physical interpretation is that solitary waves arise when dispersive and nonlinear effects are balanced, and inhomogeneity of \(n\) ensures that \(uNu^2\) is nonlinear in the long-wave limit. 

Under assumptions A1 and A2, the wave-speed estimate in \cref{eq:small_solutions_est} would instead be
\begin{equation*}
    m(0)-\nu \eqsim \mu^{\beta}, \quad \text{where}\quad \beta = \frac{s'}{s'-1}.\label{eq:small_solutions_est2}
\end{equation*}
This is the same estimate that was shown for an equation like \cref{eq:model_general} with a local nonlinearity with a cubic leading-order term in \cite{maehlen2020}. 

{
\section{Acknowledgements}
The author gratefully acknowledges the carful reading and valuable inputs from the referees that helped improve the structure of the manuscript.}
\appendix
\section{Technical tools}
\label{appendix}
\begin{proposition}[Products in \(H^t(\mathbb{R})\) {\cite[Theorem 4.6.1/1]{runst2011}}]
    Assume that \(t_1, t_2 \in \mathbb{R}\) satisfy
    \begin{equation*}
    t_1 \leq t_2 \quad \text{and} \quad t_1 + t_2 > 0.
    \end{equation*}
    Furthermore, let \(f\in H^{t_1}(\mathbb{R}), g\in H^{t_2}(\mathbb{R})\). 
    \begin{enumerate}[label = \emph{(\roman*)}]
    \item Let \(t_2> 1/2\). Then
    \[
    \norm{fg}_{H^{t_1}}\lesssim \norm{f}_{H^{t_1}}\norm{g}_{H^{t_2}}.
    \]
    \item Let \(t_2 <1/2\). Then
    \[
    \norm{fg}_{H^{t_1 + t_2 - 1/2}} \lesssim \norm{f}_{H^{t_1}}\norm{g}_{H^{t_2}}.
    \]
    \end{enumerate}
    \label{prop:product_rule_hilbert}
    \end{proposition}

{
\begin{proposition}[Products in \(L^{p}(\mathbb{R})\) {\cite[Theorem 4.4.4/3]{runst2011}}]
    Assume that \(t>0\) and that \(1<p, q<\infty\) satisfy
    \begin{equation*}
        \frac{1}{p}< \frac{1}{q} + \frac{1}{2}, \qquad \frac{1}{p}> \frac{1}{q} + \left[\frac{1}{2} - t\right]^+, \qquad t > \frac{1}{q} - \frac{1}{2},
    \end{equation*}
    where \([a]^+ = \max(a,0)\). Let \(f\in L^{q}(\mathbb{R}), g\in H^{t}(\mathbb{R})\). Then 
    \begin{equation*}
        \norm{fg}_{L^{p}}\lesssim \norm{f}_{L^{q}}\norm{g}_{H^{t}}.
    \end{equation*}
    \label{prop:product_rule_general}
\end{proposition}}

\end{document}